\documentclass[11pt,reqno]{amsart}
\usepackage[a4paper, hmargin={2.7cm,2.7cm},vmargin={3.3cm,3.3cm}]{geometry}
\usepackage[english]{babel}
\usepackage{color}
\usepackage[noadjust]{cite}
\usepackage{amsmath}
\usepackage{amssymb}
\usepackage{amsfonts}
\usepackage{amsthm}
\usepackage{enumerate}
\usepackage{amsthm}
\usepackage{bbm}

\usepackage[textwidth=20mm]{todonotes}
\usepackage{cleveref}
\usepackage{commath}

\usepackage{etoolbox}

\makeatletter
\patchcmd{\ttlh@hang}{\parindent\z@}{\parindent\z@\leavevmode}{}{}
\patchcmd{\ttlh@hang}{\noindent}{}{}{}
\makeatother

\newcommand\numberthis{\addtocounter{equation}{1}\tag{\theequation}}

\newtheorem{theorem}{Theorem}[section]
\newtheorem{lemma}[theorem]{Lemma}
\newtheorem{proposition}[theorem]{Proposition}
\newtheorem{corollary}[theorem]{Corollary}

\theoremstyle{definition}
\newtheorem{definition}[theorem]{Definition}
\newtheorem{example}[theorem]{Example}

\theoremstyle{remark}

\numberwithin{equation}{section}

\def\XXint#1#2#3{{\setbox0=\hbox{$#1{#2#3}{\int}$ }
\vcenter{\hbox{$#2#3$ }}\kern-.6\wd}}

\newcommand{\T}{\mathbb{T}}
\newcommand{\Z}{\mathbb{Z}}
\newcommand{\N}{\mathbb{N}}
\newcommand{\R}{\mathbb{R}}
\newcommand{\C}{\mathbb{C}}
\newcommand{\Q}{\mathbb{Q}}

\DeclareMathOperator{\ad}{ad}
\DeclareMathOperator{\dist}{dist}

\DeclareMathOperator{\Span}{span}

\DeclareMathOperator{\vol}{vol}

\DeclareMathOperator{\spn}{span}
\DeclareMathOperator{\clspn}{\overline{\spn}}

\newcommand{\Rspan}{\mathbb{R}-\Span}

\newcommand{\Hpi}{\mathcal{H}_{\pi}}

\title[Coherent systems over approximate lattices  in amenable groups]{Coherent systems over approximate lattices in amenable groups}

\makeatletter
\@namedef{subjclassname@2020}{\textup{2020} Mathematics Subject Classification}
\makeatother



\author{Ulrik Enstad}
\address{Department of Mathematics,
University of Oslo,
Moltke Moes vei 35,
0851 Oslo.}
\email{ubenstad@math.uio.no}

\author{Jordy Timo van Velthoven}
\address{Faculty of Mathematics,
University of Vienna,
Oskar-Morgenstern-Platz 1,
1090 Vienna, Austria}
\email{jordy-timo.van-velthoven@univie.ac.at}

\begin{document}

\maketitle

\begin{abstract}
Let $G$ be a second-countable amenable group with a uniform $k$-approximate lattice $\Lambda$. For a projective discrete series representation $(\pi, \Hpi)$ of $G$ of formal degree $d_{\pi} > 0$, we show that $D^-(\Lambda) \geq d_{\pi} / k$ is necessary for the coherent system $\pi(\Lambda) g$ to be complete in $\Hpi$. In addition, we show that if $\pi(\Lambda^2) g$ is minimal, then $D^+ (\Lambda^2) \leq d_{\pi} k$. Both necessary conditions recover sharp density theorems for uniform lattices and are new even for Gabor systems in $L^2 (\mathbb{R})$. 
As an application of the approach, we also obtain necessary density conditions for coherent frames and Riesz sequences associated to general discrete sets. All results are valid for amenable unimodular groups of possibly exponential growth.
\end{abstract}

\section{Introduction}
Let $(\pi, \Hpi)$ be an irreducible, square-integrable projective unitary representation of a second-countable amenable unimodular group $G$ with Haar measure $\mu$. For a discrete  $\Lambda \subseteq G$ and  $g \in \Hpi$, this paper concerns the relationship between the completeness of a coherent system
\begin{align} \label{eq:coherent_system}
\pi(\Lambda) g = (\pi(\lambda) g )_{ \lambda \in \Lambda }
\end{align}
in $\Hpi$, i.e.,\ the property that $\clspn \pi(\Lambda)g = \Hpi$, and the associated lower and upper Beurling densities of $\Lambda$:
\begin{align*}
    D^-(\Lambda) = \liminf_{n \to \infty} \inf_{x \in G} \frac{|\Lambda \cap xK_n|}{\mu(K_n)}, && D^+(\Lambda) = \limsup_{n \to \infty} \sup_{x \in G} \frac{|\Lambda \cap xK_n|}{\mu(K_n)} ,
\end{align*}
where $(K_n)_{n \in \N}$ is any strong Følner sequence (also called a van Hove sequence) in $G$, cf.\ \Cref{sec:beurling}. In compactly generated groups of polynomial growth, a strong Følner sequence can be obtained from sequences of balls induced by a periodic metric \cite{breuillard2014geometry}, e.g.,\ a word metric, a left-invariant Riemannian or a Carnot--Carath\'eodory metric in the case of connected Lie groups.
Beyond polynomial growth, a sequence of balls induced by a periodic metric does not need to be a F\o lner sequence, e.g.,\ the full sequence of balls $(B_r(x))_{r > 0}$ is not Følner when $G$ has exponential growth. On the other hand, any amenable group always admits a strong Følner sequence. Indeed, this is one of many equivalent characterizations of amenability.

A uniform lattice $\Lambda$ in $G$, i.e.,\ a discrete, co-compact subgroup of $G$, has uniform Beurling density $D(\Lambda) = D^-(\Lambda) = D^+(\Lambda) = 1/\vol(G/\Lambda)$, where $\vol(G/\Lambda)$ denotes the covolume of $\Lambda$. For uniform lattices, the completeness properties of coherent systems \eqref{eq:coherent_system} are fully understood. The density theorem for lattice orbits \cite{bekka2004square,romero2022density, heil2007history} asserts that if $\pi(\Lambda) g$ is complete in $\Hpi$ for some $g \in \Hpi$, then
\begin{align} \label{eq:lattice_density}
D(\Lambda) \geq d_\pi ,
\end{align}
where $d_{\pi} > 0$ denotes the formal degree of $\pi$. For classes of groups $G$, e.g., nilpotent and exponential solvable Lie groups, converse density theorems also hold, in the sense that
for any lattice $\Lambda \leq G$ satisfying  \eqref{eq:lattice_density}, there exists $g \in \Hpi$ such that $\pi(\Lambda) g$ is complete in $\Hpi$. See, e.g., 
\cite{bekka2004square, romero2022density, enstad2021density, enstad2022sufficient} for various density theorems.

In comparison to the case of uniform lattices, the completeness of coherent systems $\pi(\Lambda) g$ associated to general discrete sets $\Lambda \subseteq G$ behaves dramatically different. The showcase demonstrating this is given by so-called \emph{Gabor systems} in $L^2 (\mathbb{R})$, which are the coherent systems $\pi(\Lambda)g$ arising from the (projective) Schrödinger representation of $G = \R^2$ on $L^2(\R)$, given by 
\begin{align} \label{eq:gabor}
\pi(x,\xi)g = e^{2\pi i \xi \cdot}g(\cdot-x), \quad g \in L^2(\R) .
\end{align}
With the usual normalization of the Lebesgue measure on $\R^2$, the formal degree of $\pi$ is $d_{\pi} = 1$.

The completeness of a Gabor system $\pi(\Lambda) g$ with $g(t) = e^{-\pi t^2}$ being the Gaussian is relatively well-understood. In \cite{belov2020upper, ascensi2009phase}, it is shown that complete Gaussian Gabor systems exist with $\Lambda$ being a so-called \emph{regularly distributed set} of lower Beurling density $D^- (\Lambda) < 1$, and that for general sets $\Lambda$ the lower density could be $0$, but that $D^+ (\Lambda) > 1/(3\pi)$ is still necessary for completeness of $\pi(\Lambda) g$. However, for a complete Gabor system $\pi(\Lambda) g$ with general $g \in L^2 (\mathbb{R})$, one could have $D^- (\Lambda) = 0$ and $D^+ (\Lambda) < \varepsilon$ for arbitrary $\varepsilon > 0$ and $\Lambda$ being a subset of a lattice in $\R^2$; see \cite{Wa04}. Moreover, as shown in \cite{romero2004complete}, there exist complete Gabor systems $\pi(\Lambda) g$
with upper density $D^+ (\Lambda) = 0$.
Thus, for a general discrete set $\Lambda \subseteq \mathbb{R}^2$ and $g \in L^2 (\mathbb{R})$, there are no necessary density conditions for $\pi(\Lambda) g$ to be complete in $L^2 (\mathbb{R})$.

The discussion so far leads to a natural question: While one cannot expect necessary density conditions for complete coherent systems over arbitrary point sets, is there a class of ``regular'' point sets that contains uniform lattices for which such conditions hold? The purpose of the present paper is to consider this question for \emph{uniform approximate lattices} as introduced in \cite{bjorklund2018approximate}. In contrast to a genuine uniform lattice, a uniform approximate lattice need not be a subgroup, but only a $k$-approximate group in the sense of \cite{tao2008product, breuillard2012structure}, that is, a symmetric set $\Lambda$ containing the identity such that $\Lambda^2 \subseteq F \Lambda$ for some finite set $F \subseteq G$ of cardinality $|F| \leq k$. While a uniform $1$-approximate lattice is exactly a uniform lattice, uniform $k$-approximate lattices for $k \geq 2$ are much more general and include, among others, Meyer's model sets arising from cut-and-project schemes \cite{meyer1972algebraic,bjorklund2018spherical} and Delone sets arising from symbolic substitutions; see \cite{baake2013aperiodic} and the references therein. 

The main theorem of the present paper provides an extension of the density theorem for lattice orbits to coherent systems over uniform approximate lattices. Its statement is as follows. 

\begin{theorem} \label{thm:main1}
Let $G$ be a second-countable amenable unimodular group with a uniform $k$-approximate lattice $\Lambda \subseteq G$.
Let $(\pi, \Hpi)$ be an irreducible,  square-integrable projective representation of $G$ of formal degree $d_{\pi} > 0$. If there exists $g \in \Hpi$ such that $\pi(\Lambda) g$ is complete in $\Hpi$, then 
\begin{align} \label{eq:density_approx}
    D^- (\Lambda) \geq d_{\pi} / k,
\end{align}
where $D^- (\Lambda)$ denotes the lower Beurling density of $\Lambda$ in $G$. 
\end{theorem}

 \Cref{thm:main1} is new even for the special case of Gabor systems \eqref{eq:gabor} in $L^2 (\mathbb{R})$.
  Moreover, in contrast to \cite{bekka2004square, romero2022density, enstad2021density, enstad2022sufficient}, \Cref{thm:main1} is also applicable to groups admitting a uniform approximate lattice but no uniform lattice. 

\Cref{thm:main1} is optimal in the sense that the density $D^-(\Lambda)$ could be arbitrary small while $d_{\pi} = 1$ (see \Cref{prop:optimality}), and hence one cannot remove the denominator $k \geq 1$ in the right-hand side of \eqref{eq:density_approx}. In addition, \Cref{thm:main1} recovers  the density theorems \cite{bekka2004square, romero2022density} for genuine uniform lattices which corresponds to the case $k=1$.

For proving \Cref{thm:main1} in the case of a lattice $\Lambda \leq G$, it suffices to obtain a necessary density condition under the stronger assumption that $\pi(\Lambda) g$ is a so-called \emph{frame} (see \Cref{sec:frames_riesz}) as the existence of complete systems and frames are equivalent for lattice index sets, see, e.g., \cite{romero2022density, bekka2004square, gabardo2003frame}. In turn, the frame property of a lattice orbit allows for a relatively simple direct proof of the necessary density condition. However, the correspondence between the existence of complete systems and frames does not extend to general approximate lattices. The present paper uses therefore a different approach for obtaining \Cref{thm:main1}, namely it uses necessary density conditions for systems $\pi(\Lambda) g$ satisfying the so-called \emph{homogeneous approximation property (HAP)} introduced in \cite{RaSt95} (see \Cref{def:HAP}). The HAP is a stronger condition than merely completeness and does, in contrast to the completeness property, allow for necessary density conditions for general point sets. This is shown by the next \Cref{thm:HAP_intro}; see also \cite{RaSt95, ahn2018density} for related results in the case of Gabor systems.

\begin{theorem} \label{thm:HAP_intro}
Let $G$ be a second-countable amenable unimodular group with an irreducible, square-integrable projective representation $(\pi, \Hpi)$ of formal degree $d_{\pi} > 0$. Let $\Lambda \subseteq G$ be locally finite and $g \in \Hpi$. 

If $\pi(\Lambda) g$ satisfies the homogeneous approximation property (HAP), then $D^- (\Lambda) \geq d_{\pi}$. 
\end{theorem}

In \Cref{prop:squared_hap}, we show that for a relatively dense set $\Lambda \subseteq G$ a complete coherent system associated to the $2$-fold product $\Lambda^2$, i.e., $\pi (\Lambda^2) g$ for some $g \in \Hpi$, satisfies the HAP. In combination, this allows a proof of \Cref{thm:main1}; see the proof of \Cref{thm:density_approx} for details. Our proof method for \Cref{thm:HAP_intro} is inspired by the use of the HAP in the lattice density theorem for Gabor systems in $L^2(\mathbb{R})$ (see \cite{ahn2018density, RaSt95}) as pointed out throughout the text. 

In addition to the above results on complete systems, the present paper also provides complementary results on (uniformly) minimal systems (see Sections \ref{sec:uniformly_minimal} and \ref{sec:minimal} for definitions).

\begin{theorem} \label{thm:minimal_intro}
With notation as in \Cref{thm:HAP_intro},
\begin{enumerate}[(i)]
    \item If $\pi(\Lambda) g$ is uniformly minimal, then $D^+ (\Lambda) \leq d_{\pi}$. 
    \item If $\Lambda \subseteq G$ is a uniform $k$-approximate lattice and $\pi(\Lambda^2) g$ is minimal, then  $D^+ (\Lambda^2) \leq d_{\pi} k$.
    \end{enumerate}
\end{theorem}

In part (ii) of \Cref{thm:minimal_intro}, if $\Lambda$ is a uniform lattice (so that one can choose $k = 1$), then $D^+ (\Lambda) \leq d_{\pi}$ is necessary for the minimality of $\pi(\Lambda) g$. For arbitrary $\Lambda \subseteq G$, the upper density $D^+ (\Lambda)$ could be arbitrary large whenever $\pi (\Lambda) g$ is merely minimal, cf. \cite[Proposition 1]{ahn2018density}. 

Our proof of \Cref{thm:minimal_intro} closely follows the approach in \cite{ahn2018density} for the case of Gabor systems. For groups $G$ of polynomial growth, part (i) can also be deduced from \cite[Theorem 3.10]{mitkovski2020density}, but it is new for groups with exponential growth.

Lastly, it is worth mentioning that a direct consequence of Theorems \ref{thm:HAP_intro} and \ref{thm:minimal_intro} are the following necessary density conditions for frames and Riesz sequences (see \Cref{sec:frames_riesz}). 

\begin{corollary} \label{cor:frame_intro}
With notation as in \Cref{thm:HAP_intro} and \Cref{thm:minimal_intro},
\begin{enumerate}[(i)]
    \item If $\pi (\Lambda) g$ is a frame for $\Hpi$ and $g \in \mathcal{B}_{\pi}^2$ (see \Cref{eq:b2_vector}), then $D^- (\Lambda) \geq d_{\pi}$. 
    \item If $\pi(\Lambda) g$ is a Riesz sequence in $\Hpi$, then $D^+ (\Lambda) \leq d_{\pi}$. 
\end{enumerate}
\end{corollary}

Let us emphasize that all the above results are valid in any second-countable amenable unimodular group, showing that Beurling densities defined in terms of strong Følner sequences give the right notion of density in this generality. In contrast, the previous density theorems for coherent frames and Riesz sequences over irregular point sets \cite{fuehr2017density,mitkovski2020density} assume that the underlying group is compactly generated of polynomial growth, in which case Beurling densities can be expressed in terms of balls coming from a word metric. That Beurling densities in terms of Følner sequences give rise to density theorems for coherent systems was also shown recently in \cite{enstad2022dynamical}. In particular, a generalization of \Cref{cor:frame_intro} applicable to non-amenable unimodular groups was proved in \cite{enstad2022dynamical}, but we give a more elementary proof in the amenable case in the present paper.

The paper is organized as follows. Section \ref{sec:prelim} contains the essential background on approximate lattices and Beurling densities. All density conditions for coherent systems are proven in Section \ref{sec:density_coherent}. 
Section \ref{sec:example} consists of two examples demonstrating the main results.

\section{Approximate lattices and Beurling densities} \label{sec:prelim}

Throughout this section, $G$ denotes a second-countable locally compact group. Left Haar measure on $G$ is denoted by $\mu$. 

\subsection{Approximate lattices}
A subset $\Lambda \subseteq G$ is called \emph{locally finite} if it is closed and discrete, equivalently, $\Lambda \cap K$ is finite for every compact set $K \subseteq G$. 
Furthermore,  $\Lambda$ is said to be \emph{relatively dense} if there exists a compact set $K \subseteq G$ such that $\Lambda K = G$, and it is said to be \emph{uniformly discrete} if there exists an open set $V \subseteq G$ such that $|\Lambda \cap x V| \leq 1$ for all $x \in G$. 
A \emph{Delone set} is a set that is both relatively dense and uniformly discrete. Lastly, $\Lambda$ is said to be of \emph{finite local complexity (FLC)} if $\Lambda^{-1} \Lambda$ is locally finite. 

Following \cite{bjorklund2018approximate}, we define uniform approximate lattices as follows.

\begin{definition}
Let $k \in \mathbb{N}$. A subset $\Lambda \subseteq G$ is called a \emph{$k$-approximate subgroup} of $G$ if the following properties hold:
\begin{enumerate}
    \item[(a1)] The identity $e \in G$ is contained in $\Lambda$;
    \item[(a2)] $\Lambda$ is symmetric, i.e.,\ $\Lambda^{-1} = \Lambda$;
    \item[(a3)] there exists a finite set $F \subseteq G$ of cardinality $|F| \leq k$ such that $\Lambda^2 \subseteq F \Lambda$.
\end{enumerate}
A $k$-approximate subgroup $\Lambda$ is called a \emph{uniform $k$-approximate lattice} if it is also a Delone set. 
A set $\Lambda$ is called simply a \emph{uniform approximate lattice} if it is a uniform $k$-approximate lattice for some $k$. 
\end{definition}

Note that for symmetric $\Lambda$ the condition (a3) is equivalent to the existence of a finite set $F' \subseteq G$ such that $\Lambda^2 \subseteq \Lambda F'$ (choose $F' = F^{-1}$).

We mention the following fact, cf. \cite[Corollary 2.10]{bjorklund2018approximate}.
\begin{lemma}[\cite{bjorklund2018approximate}] \label{lem:relative_dense}
Let $\Lambda$ be a uniform approximate lattice in $G$ and let $\Gamma \subseteq \Lambda$ be a symmetric set which contains the identity and is relatively dense in $G$. Then $\Gamma$ is a uniform approximate lattice.
\end{lemma}

Any uniform lattice in $G$, i.e., a discrete, co-compact subgroup, is precisely a uniform $1$-approximate lattice in $G$. By \Cref{lem:relative_dense}, 
also relatively dense symmetric subsets of lattices containing the identity form an approximate lattice. 
Other examples of approximate lattices are so-called model sets and suitable  subsets thereof.

\begin{example}[Model sets] \label{ex:model_set}
A cut-and-project scheme is a triple $(G,H,\Gamma)$ where $G$ and $H$ are locally compact groups and $\Gamma$ is a lattice in $G \times H$ which projects injectively to $G$ and densely to $H$. Given a compact set $W \subseteq H$ of non-void interior, a set of the form
\[ \Lambda = p_G( \Gamma \cap (G \times W)) \]
where $p_G \colon G \times H \to G$ denotes the projection onto $G$ is called a \emph{model set} in $G$. The set $W$ is called the \emph{window} associated to the model set. If $\Gamma$ is a uniform lattice and $W$ is a symmetric, compact neighborhood of the identity of $H$, then $\Lambda$ is a uniform approximate lattice: Clearly $\Lambda$ contains the identity and is symmetric. Denoting by $p_H$ the projection of $G \times H$ onto $H$, we have that $p_H(\Gamma)W = H$ since $p_H(\Gamma)$ is dense in $H$. By compactness there exists a finite subset $F'$ of $\Gamma$ such that $W^2 \subseteq p_H(F')W$. Setting $F = p_G(F')$, it is not hard to see that $\Lambda^2 \subseteq F \Lambda$, hence $\Lambda$ is an approximate lattice. See also \cite[Proposition 2.13]{bjorklund2018approximate}.
\end{example}

In a locally compact abelian group $G$, any uniform approximate lattice $\Lambda$ is a subset of and commensurable to a model set $\Lambda'$ \cite{meyer1972algebraic}, i.e.,\ there exists a finite set $F \subseteq G$ such that $\Lambda' \subseteq F\Lambda$. This result has been proved for approximate lattices in connected solvable Lie groups (in particular, nilpotent Lie groups) \cite{machado2020approximate, machado2022infinite} and most recently in the general seting of approximate lattices in amenable groups \cite{machado-2023}. Uniform approximate lattices might exist in groups not admitting a uniform lattice, see \cite[Section 2.4]{bjorklund2018approximate} and \cite[Theorem 1.5]{machado2020approximate}.

In the sequel, only amenable groups $G$ (see \Cref{sec:beurling}) will be considered. If such a group admits a uniform approximate lattice, then it is necessarily unimodular. In addition, any compactly generated group $G$ admitting a uniform approximate lattice is necessarily unimodular. See \cite[Theorem 5.8]{bjorklund2018approximate} for both assertions.

\subsection{Beurling density} \label{sec:beurling}
A second-countable locally compact group $G$ is called \emph{amenable} if it admits a (right) F\o lner sequence, i.e., a sequence $(K_n)_{n \in \mathbb{N}}$ of non-null compact subsets $K_n \subseteq G$ satisfying 
\[
\lim_{n \to \infty} \frac{\mu(K_n \triangle K_n K)}{\mu(K_n)} = 0
\]
for all compact sets $K \subseteq G$. Here, $\triangle$ denotes the symmetric difference $A \triangle B = (A \setminus B) \cup (B \setminus A)$.

A \emph{(right) strong F\o lner sequence} is a F\o lner sequence $(K_n)_{n \in \N}$ satisfying the stronger condition
\[
\lim_{n \to \infty} \frac{\mu( K_n K \cap K_n^c K)}{\mu (K_n)} = 0
\]
for all compact sets $K \subseteq G$. Strong F\o lner sequences are also called van Hove sequences \cite{ornstein1987entropy,schlottmann2000generalized} and do exist in every second-countable amenable group, see, e.g., \cite{zorinkranich2014return, pogorzelski2021leptin, pogorzelski2016banach}. In fact, if $(K_n)_{n \in \mathbb{N}}$ is a F\o lner sequence and $L$ is any compact, symmetric neighborhood of the identity, then $(K_nL)_{n}$ is a strong Følner sequence, see, e.g., \cite[Proposition 5.10]{pogorzelski2021leptin}.

If $G$ is a compactly generated group of polynomial growth, i.e., given a symmetric generating neighborhood $U \subseteq G$, there exist $C, D > 0$ such that $\mu(U^n) \leq C n^D$ for all $n \in \mathbb{N}$, 
then any sequence of balls $(B_{r_n}(e) )_{n \in \N}$ with $\lim_n r_n \to \infty$ associated to a so-called \emph{periodic metric} form a strong F\o lner sequence, cf. \cite{breuillard2014geometry}. In particular, this applies to word metrics and to left-invariant Riemannian metrics and Carnot-Carath\'eodory metrics on connected Lie groups; see \cite[Example 4.3]{breuillard2014geometry} and \cite[Theorem 9]{tessera2007volume}. 

Given any strong F\o lner sequence $(K_n)_{n \in \mathbb{N}}$ in an amenable unimodular group $G$, the associated \emph{lower} and \emph{upper Beurling density} of a discrete set $\Lambda \subseteq G$ are defined by
\begin{align*}
    D^-(\Lambda) = \liminf_{n \to \infty} \inf_{x \in G} \frac{|\Lambda \cap x K_n|}{\mu(K_n)} \quad \text{and} \quad D^+(\Lambda) = \limsup_{n \to \infty} \sup_{x \in G} \frac{|\Lambda \cap x K_n|}{\mu(K_n)},
\end{align*}
respectively. The densities $D^-$ and $D^+$ are independent of the choice of strong  F\o lner sequence (cf.\ \cite[Proposition 5.14]{pogorzelski2021leptin}). 

A uniform lattice $\Lambda \subseteq G$ has Beurling density given by
\[
D^- (\Lambda) = D^+ (\Lambda) = \vol(G/ \Lambda)^{-1}.
\]
More generally, if $\Lambda \subseteq G$ is a model set coming from a cut-and-project scheme $(G,H,\Gamma)$ with a uniform lattice $\Gamma \leq G \times H$ and window $W \subseteq H$ (cf.\ \Cref{ex:model_set}), then
\[ D^-(\Lambda) = D^+(\Lambda) = \frac{\mu_H(W)}{\vol((G\times H)/\Gamma)} , \]
where $\mu_H$ denotes Haar measure on $H$, cf.\ \cite[Theorem 7.2]{pogorzelski2021leptin}. In general, however, approximate lattices need not have coinciding lower and upper Beurling density, see, e.g.,\ \cite[Example 4.15]{bjorklund2018approximate}.

\section{Density conditions for coherent systems} \label{sec:density_coherent} 

This section provides density conditions for various reproducing properties of coherent systems. Henceforth, it will additionally be assumed that $G$ is amenable and unimodular.

\subsection{Coherent systems}

A \emph{projective unitary representation} of $G$ is a strongly continuous map $\pi : G \to \mathcal{U}(\Hpi)$ on a Hilbert space $\Hpi$ such that
\[ \pi(x)\pi(y) = \sigma(x,y)\pi(xy) , \quad x,y \in G ,\]
where $\sigma \colon G \times G \to \T$ is an associated (necessarily measurable) function called a \emph{2-cocycle}. A projective representation $\pi$ is called \emph{irreducible} if $\{ 0 \}$ and $\Hpi$ are the only $\pi(G)$-invariant closed subspaces of $\Hpi$.

 An irreducible projective unitary representation is called a \emph{discrete series} if there exists a non-zero $g \in \Hpi$ such $\int_G | \langle g, \pi(x) g \rangle |^2 d\mu (x) < \infty$. In that case, there exists $d_\pi > 0$, called the \emph{formal degree} of $\pi$, satisfying
\begin{align} \label{eq:ortho}
\int_G | \langle f, \pi(x) g \rangle |^2 \; d\mu (x) = d_{\pi}^{-1} \| f \|_{\Hpi}^2 \| g \|_{\Hpi}^2, 
\end{align}
for all $f,g \in \Hpi$, see, e.g.,\ \cite{dixmier1977cstar,gaal1973}.

Given $g \in \Hpi$ and $\Lambda \subseteq G$, we refer to a family of vectors in $\Hpi$ of the form
\[ \pi(\Lambda) g = ( \pi(\lambda)g )_{\lambda \in \Lambda} \]
as a \emph{coherent system} in $\Hpi$. A coherent system is allowed to contain repeating elements. 

\subsection{The homogeneous approximation property}
A coherent system $\pi(\Lambda) g$ is said to be \emph{complete}  if $\clspn \pi(\Lambda)g = \Hpi$. Using the notation $\dist(f,S) = \inf_{g \in S}\| f - g \|$ for subsets $S \subseteq \Hpi$, it follows that completeness of $\pi(\Lambda)g$ is equivalent to the following: For every $f \in \Hpi$ and every $\varepsilon > 0$, there exists a compact unit neighborhood $K \subseteq G$ such that
\[ \dist \big(f, \; \spn \{ \pi (\lambda) g : \lambda \in \Lambda \cap K \} \big) < \varepsilon .\]

The following stronger form of completeness will be central in the arguments of this section. 
For applications to frames (see \Cref{sec:frames_riesz}), it will be convenient to consider general systems $( g_{\lambda} )_{\lambda \in \Lambda}$ of vectors $g_{\lambda} \in \Hpi$ and not just coherent systems $\pi(\Lambda) g$. 

\begin{definition} \label{def:HAP}
Let $\Lambda \subseteq G$ be countable. A family $( g_{\lambda} )_{\lambda \in \Lambda}$ of vectors $g_{\lambda} \in \Hpi$ satisfies the \emph{homogeneous approximation property (HAP)} if for every $f \in \Hpi$ and every $\varepsilon > 0$ there exists a compact unit neighborhood $K \subseteq G$ such that
\begin{align} \label{eq:HAP}
    \dist \big( \pi(x)f, \; \spn \{ g_{\lambda} : \lambda \in \Lambda \cap x K \} \big) < \varepsilon 
\end{align}
for all $x \in G$.
\end{definition}

The following theorem provides a necessary density condition for systems satisfying the HAP. 

\begin{theorem}\label{thm:HAP_density}
Let $\Lambda \subseteq G$ be locally finite and suppose that $( g_{\lambda} )_{\lambda \in \Lambda}$ satisfies the homogeneous approximation property \eqref{eq:HAP}. Then
\[ D^-(\Lambda) \geq d_{\pi} .\]
\end{theorem}

\begin{proof}
Let $\eta \in \Hpi$ be a unit vector, let $\varepsilon > 0$ and let $(K_n)_{n \in \mathbb{N}}$ be a strong F\o lner sequence in $G$. Since $\Lambda$ is locally finite, $V_{x, n} := \Span \{ g_{\lambda} : \lambda \in \Lambda \cap x K_n \}$ is a finite-dimensional subspace of $\Hpi$ for every $x \in G$ and $n \in \N$.

By assumption, there exists a symmetric compact unit neighborhood $K  \subseteq G$ such that, for all $y \in G$,
\[\dist \big(\pi(y) \eta, \; \Span \{ g_{\lambda} : \lambda \in \Lambda \cap y K \} \big) \leq \sqrt{\varepsilon}. \]
 Note that if $y \in x(K_n \setminus K_n^c K)$, then $y K \subseteq x K_n$. Indeed, since $x^{-1} y \in K_n$ but $x^{-1} y \notin K_n^c K$, it follows that $x^{-1} y \notin K_n^c k$ for all $k \in K$, and  $ y k^{-1} \in x K_n$ for all $k \in K$, so that $yK \subseteq x K_n$ by symmetry of $K$.
Therefore, 
\[
\dist \big(\pi(y) \eta, \; V_{x, n}\big) \leq \dist \big( \pi(y) \eta, \; \Span \{ g_{\lambda} : \lambda \in \Lambda \cap yK \} \big) \leq \sqrt{\varepsilon}, \quad y \in x(K_n \setminus K_n^c K).
\]
Denote by $P$ the orthogonal projection onto $V_{x,n}$. Then
\[
\| P \pi (y) \eta \|^2_{\Hpi} \geq \| \pi(y) \eta \|_{\Hpi}^2 - \| (I - P) \pi(y) \eta \|_{\Hpi}^2 \geq 1 - \varepsilon.
\]
Integrating this inequality over $x(K_n \setminus K_n^c K)$  gives
\[
(1-\varepsilon) \mu (K_n \setminus K_n^c K) \leq \int_{x( K_n \setminus K_n^c K)} \| P \pi(y) \eta \|_{\Hpi}^2 \; d\mu (y) \leq \int_{G} \| P \pi(y) \eta\|_{\Hpi}^2 \; d\mu (y)
\]
by left-invariance of Haar measure. 

For further estimating the right-hand side, let $(h_i)_{i = 1}^N$ be an orthonormal basis for $V_{x,n}$. Then, the orthogonality relations \eqref{eq:ortho} for $\pi$ and the fact that $\eta$ is of unit norm, yield
\[
\int_G \| P \pi(y) \eta \|_{\Hpi}^2 \; d\mu (y)  = \sum_{i = 1}^N \int_{G}  |\langle h_i, \pi(y) \eta \rangle |^2 \; d\mu (y) = d_{\pi}^{-1} \sum_{i = 1}^N \| h_i \|_{\Hpi}^2 = d_{\pi}^{-1} \dim V_{ x,n}. 
\]
Combining the above gives
\[
(1-\varepsilon) \mu (K_n \setminus K_n^c K) \leq d_{\pi}^{-1} \dim V_{x,n} \leq d_{\pi}^{-1} |\Lambda \cap x K_n|,
\]
and hence
\[
(1 - \varepsilon) d_{\pi} \frac{\mu (K_n \setminus K_n^c K)}{\mu( K_n )} \leq  \frac{ |\Lambda \cap x K_n |}{\mu (K_n)}. 
\]
As $x \in G$ and $\varepsilon > 0$ were chosen arbitrary, this shows that
\[
\inf_{x \in G} \frac{|\Lambda \cap x K_n|}{ \mu(K_n)} \geq d_{\pi} \frac{\mu (K_n \setminus K_n^c K)}{\mu(K_n)}. 
\]
Since
\[ \frac{\mu (K_n \setminus K_n^cK)}{\mu (K_n)} = \frac{\mu (K_n \setminus (K_n^cK \cap K_n))}{\mu (K_n)} = 1 - \frac{\mu (K_n^cK \cap K_n)}{\mu (K_n)} \geq 1 - \frac{\mu(K_n^cK \cap K_nK)}{\mu(K_n)} \to 1 \]
as $n \to \infty$, it follows that $D^- (\Lambda) \geq d_{\pi}$, as required.
\end{proof}

The above result is an extension of \cite[Theorem 3(2)]{ahn2018density} for $G = \mathbb{R}^2$ to arbitrary unimodular, amenable groups, and follows the proof outline closely.
See also \cite[Theorem 4]{nitzan2012revisiting} and \cite[Theorem 1]{RaSt95} for related proof techniques.

\subsection{Complete systems}

The theme of this subsection is the interaction between completeness and the homogeneous approximation property for coherent systems.

\begin{proposition}\label{prop:squared_hap}
Let $\Lambda \subseteq G$ be relatively dense and $g \in \Hpi$. Suppose that $\pi(\Lambda)g$ is complete. Then $\pi(\Lambda^2)g$ satisfies the homogeneous approximation property \eqref{eq:HAP}.
\end{proposition}

\begin{proof}
Since $\Lambda$ is relatively dense, there exists a compact set $K \subseteq G$ such that $G = \Lambda K$.  Let $f \in \Hpi$ and $\varepsilon > 0$.
\\\\
\textbf{Step 1.} In this step, we show that there exists a compact set $K'$ such that
\begin{equation}
    \dist(\pi(y)f, \spn \pi(\Lambda \cap K')g) < \varepsilon \;\;\; \text{for all $y \in K$.} \label{eq:compact_hap}
\end{equation}
By the strong continuity of $\pi$, there exists for each $y \in K$ an open neighborhood $U_y \subseteq G$ of $y$ such that $\| \pi(y)f - \pi(z)f \| < \varepsilon / 2$ for each $z \in U_y$. Since the sets $U_y$ cover $K$, it follows by compactness that there exist $y_1, \ldots, y_k \in K$ such that $K \subseteq U_{y_1} \cup \cdots \cup U_{y_k}$.

Since $\pi(\Lambda) g$ is complete, there exists for each $y \in K$ a compact unit neighborhood $K'_y$ such that
\[ \dist( \pi(y)f, \spn \pi(\Lambda \cap K'_y)g ) < \frac{\varepsilon}{2} .\]
Set $K' = \bigcup_{1 \leq j \leq k} K'_{y_j}$. If $y \in K$, then $y \in U_{y_j}$ for some $1 \leq j \leq k$, and hence
\begin{align*}
    \dist(\pi(y)f, \spn \pi(\Lambda \cap K')g ) &\leq \| \pi(y)f - \pi(y_j)f \| + \dist(\pi(y_j)f, \spn \pi(\Lambda \cap K')g) \\
    &< \frac{\varepsilon}{2} + \dist(\pi(y_j)f, \spn \pi(\Lambda \cap K_{y_j})g) \\
    &< \frac{\varepsilon}{2} + \frac{\varepsilon}{2} = \varepsilon .
\end{align*}
This establishes \eqref{eq:compact_hap}.
\\\\
\textbf{Step 2.} 
Let $x \in G$ and write $x = \lambda y$ for some $\lambda \in \Lambda$ and $y \in K$. By \eqref{eq:compact_hap}, there exist $c_1, \ldots, c_n \in \C$ and $\lambda_1, \ldots, \lambda_n \in \Lambda \cap K'$ such that
\[ \Big\| \pi(y)f - \sum_{i=1}^n c_i \pi(\lambda_i) g \Big\| < \varepsilon .\]
Applying the unitary operator $\pi(\lambda)$ inside the norm yields
\begin{align*}
    \Big\| \pi(y) f - \sum_{i=1}^n c_i \pi(\lambda_i) g \Big\| &= \Big\| \pi(\lambda) \pi(y) f - \sum_{i=1}^n c_i \pi(\lambda) \pi(\lambda_i)g \Big\| \\
    &= \Big\| \sigma(\lambda, y) \pi(x)f - \sum_{i=1}^n c_i \sigma(\lambda,\lambda_i) \pi(\lambda \lambda_i) g \Big\| .
\end{align*}
Hence,
\[ \Big\| \pi(x) f - \sum_{i=1}^n c_i \overline{\sigma(\lambda,y)} \sigma(\lambda,\lambda_i) \pi(\lambda \lambda_i)g \Big\| < \varepsilon . \]
Note that $\lambda \lambda_i \in \Lambda^2$ and that $\lambda \lambda_i = x y^{-1} \lambda_i \in x K^{-1} K'$. Therefore, setting $K_0 = K^{-1} K'$ gives
\[ \dist( \pi(x)f, \spn \pi(\Lambda^2 \cap x K_0)) ) < \varepsilon, \]
which finishes the proof.
\end{proof}

 \Cref{prop:squared_hap} is an adaption of the proof of \cite[Theorem 2]{RaSt95} to possibly non-subgroup index sets. An immediate consequence is the following.

\begin{corollary}
If $\Lambda$ is a uniform lattice in $G$ and $\pi(\Lambda)g$ is complete for $g \in \Hpi$, then $\pi(\Lambda)g$ satisfies the homogeneous approximation property \eqref{eq:HAP}.
\end{corollary}

\begin{proof}
Since $\Lambda$ is uniform, it is relatively dense. An application of \Cref{prop:squared_hap} then gives that $\pi(\Lambda^2)g$ satisfies the homogeneous approximation property. Since $\Lambda$ is a lattice, we have $\Lambda^2=\Lambda$, so the conclusion follows.
\end{proof}

The following theorem is the main result of this paper. It provides an extension of the density theorems \cite{bekka2004square, romero2022density} for complete lattice orbits to arbitrary approximate lattices.

\begin{theorem}\label{thm:density_approx}
Let $\Lambda \subseteq G$ be a uniform $k$-approximate lattice. If there exists $g \in \Hpi$ such that $\pi(\Lambda)g$ is complete, then
\[ D^-(\Lambda) \geq \frac{d_{\pi}}{k} .\]
\end{theorem}

\begin{proof}
Let $F \subseteq G$ be a finite set of cardinality $|F| \leq k$ such that $\Lambda^2 \subseteq \Lambda F$. Since $\Lambda$ is relatively dense and $\pi(\Lambda)g$ is complete, it follows from \Cref{prop:squared_hap} that $\pi(\Lambda^2)g$ satisfies the HAP. Since $\Lambda$ is an approximate lattice, $\Lambda$ has finite local complexity, so the point set $\Lambda^2 = \Lambda^{-1}\Lambda$ is locally finite. Hence, \Cref{thm:HAP_density} applies, which combined with $\Lambda^2 \subseteq \Lambda F$ yields
\[ d_{\pi} \leq D^-(\Lambda^2) \leq D^-(\Lambda F) .\]
Therefore, since $|F| \leq k$, it remains to show that $D^-(\Lambda F) \leq |F|D^-(\Lambda)$.

Let $(K_n)_{n \in \mathbb{N}}$ be a strong F\o lner sequence and let $K \subseteq G$ be a compact symmetric unit neighborhood such that $F \subseteq K$. Then, for fixed $y \in F$, $x \in G$ and $n \in \mathbb{N}$, 
\[
|\Lambda y \cap x K_n | = |\Lambda \cap x K_n y^{-1}| \leq |\Lambda \cap x K_n K  |. 
\]
Therefore,
\begin{align*}
    D^- (\Lambda F) &\leq \liminf_{n \to \infty} \inf_{x\in G} \sum_{y \in F} \frac{| \Lambda y \cap x K_n|}{\mu (K_n)} \\
    &\leq |F| \liminf_{n \to \infty} \inf_{x \in G} \frac{|\Lambda \cap x K_n K|}{\mu (K_n K)} \cdot \frac{\mu (K_n K)}{\mu (K_n)} \\
    &\leq |F|  D^-(\Lambda),
\end{align*}
where we in the last equality used that
\[ \frac{\mu(K_nK)}{\mu(K_n)} = \frac{\mu(K_nK \cap K_n) + \mu(K_nK \cap K_n^c)}{\mu(K_n)} \leq 1 + \frac{\mu(K_nK \cap K_n^cK)}{\mu(K_n)} \to 1 , \quad n \to \infty . \]

In conclusion, the above estimates show that
\[ d_{\pi} \leq D^-(\Lambda^2) \leq D^-(\Lambda F) \leq |F| D^-(\Lambda), \]
as required.
\end{proof}

To conclude this subsection, we illustrate the optimality of \Cref{thm:density_approx}. Let $G = \R^2$ and let $\pi$ be the Schrödinger representation of $\R^2$ on $L^2(\R)$, i.e.,\
\[ \pi(x,\xi)f(t) = e^{2\pi i \xi t} f(t-x), \quad (x,\xi) \in G, \; f \in L^2(\R) .\]
In \cite[Theorem 1]{Wa04}, it is shown that:

\begin{theorem}[\cite{Wa04}] \label{thm:wang}
Let $\Gamma$ be a lattice in $\R^2$ with $D(\Gamma) > 1$. For every $\varepsilon > 0$, there exists $g \in L^2(\R)$ and a subset $\Lambda \subseteq \Gamma$ such that $\pi(\Lambda)g$ is complete in $L^2 (\mathbb{R})$, $D^-(\Lambda) = 0$ and $D^+(\Lambda) < \varepsilon$.
\end{theorem}

By a slight modification of the example in \Cref{thm:wang}, we show the following:

\begin{proposition} \label{prop:optimality}
For every $ \varepsilon > 0$,  there exists a uniform approximate lattice $\Lambda' \subseteq \R^2$ and $g \in L^2(\R)$ such that $\pi(\Lambda')g$ is complete in $L^2 (\mathbb{R})$ and $D^+(\Lambda') < \varepsilon$ (so that in particular $D^-(\Lambda') < \varepsilon$).
\end{proposition}

\begin{proof}
Let $n > 1$ be an integer such that $1/n < \varepsilon$ and consider the lattice $\Gamma = n^{-1}\Z \times \Z$. Since $D(\Gamma) = n > 1$, \Cref{thm:wang} yields a set $\Lambda \subseteq \Gamma$ and a vector $g \in L^2(\R)$ such that $\pi(\Lambda)g$ is complete, $D^-(\Lambda) = 0$ and $D^+(\Lambda) < \varepsilon' := (\varepsilon - 1/n)/2$. Set
\[ \Lambda' = \Lambda \cup (-\Lambda) \cup (n\Z \times \Z) .\]
Then $\Lambda'$ is symmetric, contains the identity, and is a subset of $\Gamma$. Using the subadditivity of the upper Beurling density, we obtain the following chain of inequalities:
\begin{align*}
    0 &< \frac{1}{n} = D^-(n\Z \times \Z) \leq D^-(\Lambda') \leq D^+(\Lambda') \\
    &\leq D^+(\Lambda) + D^+(-\Lambda) + D^+(n\Z \times \Z) \\
    &< 2 \varepsilon' + \frac{1}{n} = \varepsilon .
\end{align*}
The positivity of $D^-(\Lambda')$ is equivalent to the relative density of $\Lambda'$ in $G$. Thus, $\Lambda'$ is a uniform approximate lattice by \Cref{lem:relative_dense}.
\end{proof}

\subsection{Uniformly minimal systems} \label{sec:uniformly_minimal}

A coherent system $\pi(\Lambda) g$ is called \emph{minimal} if
\[ \pi(\lambda)g \notin \clspn \{ \pi(\lambda')g : \lambda' \in \Lambda \setminus \{ \lambda \} \} \]
for every $\lambda \in \Lambda$. Equivalently,  $\pi(\Lambda) g$ admits a \emph{bi-orthogonal system}, i.e.,\ a sequence $(h_\lambda )_{\lambda \in \Lambda}$ such that $\langle \pi(\lambda)g, h_{\lambda'} \rangle = \delta_{\lambda,\lambda'}$ for every $\lambda,\lambda' \in \Lambda$.

A stronger property is the following.

\begin{definition}
The coherent system $\pi(\Lambda) g$ is called \emph{uniformly minimal} if there exists $\delta > 0$ such that
\[ \dist\Big( \pi(\lambda)g, \, \clspn \{ \pi(\lambda')g : \lambda' \in \Lambda \setminus \{ \lambda \} \} \Big) \geq \delta \]
for every $\lambda \in \Lambda$.
\end{definition}

The uniform minimality of $\pi(\Lambda) g$ is equivalent to the existence of a \emph{bounded} bi-orthogonal system, i.e.,\ a bi-orthogonal system $(h_\lambda )_{\lambda \in \Lambda}$ such that $\sup_{\lambda \in \Lambda} \| h_\lambda \| < \infty$, see, e.g.,\ \cite[Lemma 6]{ahn2018density}.

\begin{lemma}\label{lem:uniform_minimality_separated}
If $\pi(\Lambda)g$ is uniformly minimal, then $\Lambda$ is uniformly discrete.
\end{lemma}

\begin{proof}
Let $\delta > 0$ be such that
\[ \dist\Big( \pi(\lambda)g, \, \clspn \{ \pi(\lambda')g : \lambda' \in \Lambda \setminus \{ \lambda \} \} \Big) \geq \delta \]
for every $\lambda \in \Lambda$. By the strong continuity of $\pi$, there exists a unit neighborhood $U$ such that $\| \pi(x)g - \pi(x')g \| < \delta$ whenever $x^{-1}x' \in U$. Let $V$ be a symmetric neighborhood of the identity with $V^2 \subseteq U$. Assume towards a contradiction that there exist distinct $\lambda,\lambda' \in \Lambda \cap xV$ for some $x \in G$. Then $(x^{-1}\lambda)^{-1}(x^{-1}\lambda') = \lambda^{-1}\lambda' \in V^2 \subseteq U$, so $\| \pi(\lambda)g - \pi(\lambda')g \| < \delta$. But since $\lambda' \in \Lambda \setminus \{ \lambda \}$ we also get
\[ \| \pi(\lambda) g - \pi(\lambda')g \| \geq \dist\Big( \pi(\lambda)g, \, \clspn \{ \pi(\lambda')g : \lambda' \in \Lambda \setminus \{ \lambda \} \} \Big) \geq \delta , \]
a contradiction.
\end{proof}

\begin{theorem} \label{thm:uniformly_minimal}
Let $\Lambda \subseteq G$ and suppose that $\pi(\Lambda) g$ is uniformly minimal for some $g \in \Hpi$. Then 
\[
D^+ (\Lambda) \leq d_{\pi}.
\]
\end{theorem}
\begin{proof}
Let $\eta \in \Hpi$ be a unit vector and define $V_{\eta} : \Hpi \to L^2 (G)$ by $V_\eta f = \langle f, \pi (\cdot) \eta \rangle$. 
Fix $\varepsilon > 0$ and choose a compact symmetric unit neighborhood $K \subseteq G$  such that 
\[
\int_{G\setminus K} |V_{\eta} g (y)|^2 \; d\mu(y) \leq \varepsilon^2.
\]
Let $(K_n)_{n \in \mathbb{N}}$ be a strong F\o lner sequence. For fixed $x \in G$ and $n \in \mathbb{N}$, let \[ V_{x,n} := \Span \big\{ \pi(\lambda) g : \lambda \in \Lambda \cap x K_n \big\}. \]
Then, by \Cref{lem:uniform_minimality_separated}, the space $V_{x,n}$ is finite-dimensional since $\Lambda$ is uniformly discrete (in particular, locally finite). Since $\pi(\Lambda) g$ is uniformly minimal, there exists a bi-orthogonal system $(h_{\lambda})_{\lambda \in \Lambda}$ in $\Hpi$ satisfying $B := \sup_{\lambda \in \Lambda} \| h_{\lambda} \|_{\Hpi}^2 < \infty$. 
If $P$ denotes the projection onto $V_{x,n}$, then $\pi(\Lambda \cap x K_n ) g$ and $( P h_{\lambda} )_{\lambda \in \Lambda \cap x K_n }$ are bi-orthogonal, 
since $\langle \pi(\lambda) g, P h_{\lambda'} \rangle = \langle \pi(\lambda) g, h_{\lambda'} \rangle = \delta_{\lambda, \lambda'}$ for $\lambda, \lambda' \in \Lambda \cap x K_n$. Therefore, for $y \in G$,
\begin{align*}
\sum_{\lambda \in \Lambda \cap x K_n} (V_{\eta} \pi(\lambda) g) (y) \overline{(V_{\eta} P h_{\lambda})(y)} &= \sum_{\lambda \in \Lambda \cap  x K_n} \langle \pi(\lambda) g , \pi(y) \eta \rangle \langle  \pi(y) \eta , P h_{\lambda} \rangle \\
&= \bigg \langle \sum_{\lambda \in \Lambda \cap x K_n} \langle \pi(y) \eta, P h_{\lambda} \rangle \pi(\lambda) g , \; \pi(y) \eta \bigg \rangle \\
&= \big \langle P \pi (y) \eta, \pi(y) \eta \big \rangle \\
&= \| P \pi(y) \eta \|^2_{\Hpi}.
\end{align*}
 Integrating this identity over $x K_n K$ gives
\begin{align*}
\sum_{\lambda \in \Lambda \cap x K_n} \int_{x K_n K} (V_{\eta} \pi(\lambda) g) (y) \overline{(V_{\eta} P h_{\lambda})(y)} \; d\mu(y) &= \int_{x K_n K} \| P \pi (y) \eta \|^2_{\Hpi} \; d\mu(y) \\
&\leq \mu(K_n K). 
\end{align*}
Note that the orthogonality relations \eqref{eq:ortho} yield
\[
\int_G (V_{\eta} \pi(\lambda) g) (y) \overline{(V_{\eta} P h_{\lambda})(y)} \; d\mu(y) = d_{\pi}^{-1} \langle \eta, \eta \rangle \langle \pi(\lambda ) g , P h_{\lambda} \rangle = d_{\pi}^{-1}. 
\]
This, together with an application of the triangle inequality, gives
\begin{align*}
    \sum_{\lambda \in \Lambda \cap x K_n} d_{\pi}^{-1} &\leq \bigg| \sum_{\lambda \in \Lambda \cap x K_n} \int_{x K_n K} (V_{\eta} \pi(\lambda) g) (y) \overline{(V_{\eta} P h_{\lambda})(y)} \; d\mu(y) \bigg| \\
    &\quad \quad \quad + \bigg| \sum_{\lambda \in \Lambda \cap x K_n} \int_{G \setminus x K_n K} (V_{\eta} \pi(\lambda) g) (y) \overline{(V_{\eta} P h_{\lambda})(y)} \; d\mu(y) \bigg| \\
    &\leq \mu(K_n K) +  \sum_{\lambda \in \Lambda \cap x K_n} \bigg| \int_{G \setminus x K_n K} (V_{\eta} \pi(\lambda) g) (y) \overline{(V_{\eta} P h_{\lambda})(y)} \; d\mu(y) \bigg|,
\end{align*}
which implies that
\begin{align} \label{eq:reverse_triangle}
   d_{\pi}^{-1} |\Lambda \cap x K_n| - \sum_{\lambda \in \Lambda \cap x K_n} \bigg| \int_{G \setminus x K_n K} (V_{\eta} \pi(\lambda) g) (y) \overline{(V_{\eta} P h_{\lambda})(y)} \; d\mu(y) \bigg| \leq \mu(K_n K). 
\end{align}
Since $\| V_{\eta} P {h_{\lambda}} \|_{L^2}^2 = d_{\pi}^{-1} \| P h_{\lambda} \|_{\Hpi}^2 \leq d_{\pi}^{-1} B$, each summand on the left-hand side can be estimated using Cauchy--Schwarz inequality as
\begin{align*}
\bigg|  \int_{G \setminus x K_n K} (V_{\eta} \pi(\lambda) g) (y) \overline{(V_{\eta} P h_{\lambda})(y)} \; d\mu(y) \bigg| &\leq \big\| V_{\eta} P h_{\lambda} \big\|_{L^2}  \bigg( \int_{G \setminus x K_n K} |V_{\eta} g( \lambda^{-1} y)|^2 \; d\mu(y) \bigg)^{1/2} \\
&\leq d_{\pi}^{-1/2} B^{1/2} \bigg( \int_{G \setminus \lambda^{-1} x K_n K} |V_{\eta} g(y)|^2 \; d\mu(y) \bigg)^{1/2} \\
&\leq d_{\pi}^{-1/2} B^{1/2} \bigg( \int_{G \setminus K} |V_{\eta} g(y)|^2 \; d\mu(y) \bigg)^{1/2} \\
&\leq d_{\pi}^{-1/2} B^{1/2} \varepsilon, \numberthis \label{eq:penultimate}
\end{align*}
where the penultimate inequality used that $G \setminus \lambda^{-1} x K_n K \subseteq G \setminus K$ as $\lambda \in x K_n$.
Combining inequalities \eqref{eq:reverse_triangle} and \eqref{eq:penultimate} gives
\[
d_{\pi}^{-1}  \frac{|\Lambda \cap  x K_n |}{\mu(K_n)} - d_{\pi}^{-1/2} B^{1/2} \varepsilon \frac{ |\Lambda \cap x K_n |}{\mu(K_n)} \leq \frac{\mu(K_n K)}{\mu(K_n)}.
\]
As in the proof of \Cref{thm:density_approx}, it follows that $\mu(K_n K) / \mu(K_n) \to 1$ as $n \to \infty$, 
and hence
\[
( 1 -  d_{\pi}^{1/2} B^{1/2} \varepsilon ) D^{+} (\Lambda)  \leq d_{\pi}.
\]
Since $\varepsilon > 0$ was chosen arbitrary, this completes the proof.
\end{proof}

The proof of the previous theorem follows the corresponding proofs in \cite{ahn2018density, mitkovski2020density} closely, while extending it to possibly exponential growth groups.

\subsection{Minimal systems} \label{sec:minimal}

\begin{lemma} \label{lem:minimal_uniform}
If $\Lambda \subseteq G$ is a symmetric discrete set and $\pi(\Lambda^2) g$ is minimal, then $\pi(\Lambda) g$
is uniformly minimal. 
\end{lemma}
\begin{proof}
Since $\pi(\Lambda^2) g$ is minimal, there exists $h \in \Hpi$ such that $\langle \pi(\lambda) g, h \rangle = \delta_{\lambda, e}$ for $\lambda \in \Lambda^2$. Let $\lambda_1,\lambda_2 \in \Lambda$. Since  $\lambda_2^{-1}\lambda_1 \in \Lambda^{-1} \Lambda = \Lambda^2$, a direct calculation gives
\begin{align*}
\big \langle \pi(\lambda_1) g, \pi(\lambda_2) h \big \rangle 
&= \overline{\sigma(\lambda_2, (\lambda_2)^{-1}) } \sigma((\lambda_2)^{-1}, \lambda_1) \big \langle \pi((\lambda_2)^{-1} \lambda_1) g, h \big \rangle \\
&=  \overline{\sigma(\lambda_2, (\lambda_2)^{-1}) } \sigma((\lambda_2)^{-1}, \lambda_1) \; \delta_{\lambda_2,\lambda_1} .
\end{align*}
Hence, $\langle \pi(\lambda_1) g, \pi(\lambda_2) h \big \rangle = 0$ for $\lambda_1 \neq \lambda_2$. If $\lambda_1 = \lambda_2$, then 
\[ \overline{\sigma(\lambda_2, (\lambda_2)^{-1}) } \sigma((\lambda_2)^{-1}, \lambda_1)  = \overline{\sigma(\lambda_2, (\lambda_2)^{-1}) }  \sigma(\lambda_2, (\lambda_2)^{-1}) = 1, \]
so that $\langle \pi(\lambda_1) g, \pi(\lambda_2) h \big \rangle = 1$. This shows that $\pi(\Lambda) g$ and $\pi(\Lambda) h$ are bi-orthogonal, which implies that $\pi (\Lambda) g$ is uniformly minimal. 
\end{proof}

\begin{theorem}
Let $\Lambda \subseteq G$ be a uniform $k$-approximate lattice. If $\pi(\Lambda^2) g$ is minimal, then 
\[
D^+(\Lambda^2) \leq d_{\pi} k.
\]
In particular, if $\Lambda \leq G$ is a uniform lattice, then $D^+(\Lambda) \leq d_{\pi}$. 
\end{theorem}
\begin{proof}
If $\pi(\Lambda^2) g$ is minimal, then $\pi(\Lambda) g$ is uniformly minimal by \Cref{lem:minimal_uniform}. Hence, an application of \Cref{thm:uniformly_minimal} implies that $D^+ (\Lambda) \leq d_{\pi}$. If $F \subseteq G$ is a finite set of cardinality $|F| \leq k$ satisfying $\Lambda^2 \subseteq F \Lambda$, then the subadditivity of the upper Beurling density gives
\[
D^+ (\Lambda^2) \leq |F| D^+ (\Lambda) \leq d_{\pi} k,
\]
which is the desired result.
\end{proof}

\subsection{Frames and Riesz sequences} \label{sec:frames_riesz}
A coherent system $\pi(\Lambda) g$ is a \emph{frame} for $\Hpi$ if there exist $A, B > 0$ such that
\begin{align} \label{eq:frame}
A \| f \|_{\Hpi}^2 \leq \sum_{\lambda \in \Lambda} |\langle f, \pi(\lambda) g \rangle |^2 \leq B \| f \|_{\Hpi}^2, \quad f \in \Hpi.
\end{align}
Equivalently, the frame operator $S = \sum_{\lambda \in \Lambda} \langle \cdot , \pi(\lambda) g \rangle \pi(\lambda) g$ is bounded and invertible on $\Hpi$. If $\pi(\Lambda) g$ is a frame, then $S^{-1} \pi(\Lambda) g$ is also a frame for $\Hpi$, the so-called \emph{canonical dual frame}. Clearly, any frame for $\Hpi$ is complete in $\Hpi$.

Dual to the notion of a frame, a coherent system $\pi(\Lambda) g$ is called a \emph{Riesz sequence} in $\Hpi$ if there exist $A, B >0$ such that 
\begin{align} \label{eq:riesz}
A \| c \|_{\ell^2}^2 \leq \bigg\| \sum_{\lambda \in \Lambda} c_{\lambda} \pi(\lambda) g \bigg\|_{\Hpi}^2 \leq B \| c \|_{\ell^2}^2, \quad c \in \ell^2 (\Lambda). 
\end{align}
A Riesz sequence $\pi(\Lambda) g$ is a frame for its span, and its canonical dual frame in $\overline{\Span \pi(\Lambda) g}$ is bi-orthogonal to $\pi(\Lambda) g$. In particular, any Riesz sequence is uniformly minimal. 

A coherent system $\pi(\Lambda) g$ satisfying the upper bound in \eqref{eq:frame} (equivalently, in \eqref{eq:riesz}) is called a \emph{Bessel sequence}. In this case, the index set $\Lambda$ satisfies $\sup_{x \in G} |\Lambda \cap x K| < \infty$ for some (all) compact unit neighborhoods $K \subseteq G$ with non-empty interior. In particular, $\Lambda$ must be locally finite.

In order to obtain necessary density conditions for frames (resp. Riesz sequences) $\pi(\Lambda) g$ from \Cref{thm:HAP_density} (resp. \Cref{thm:uniformly_minimal}), it remains to show that a frame satisfies the homogeneous approximation property. For this, define, for a fixed relatively compact unit neighborhood $Q \subseteq G$, the collection
\begin{align} \label{eq:b2_vector}
\mathcal{B}^2_{\pi} := \bigg\{ g \in \Hpi : \int_G | \sup_{y \in Q}| \langle f, \pi (xy) g \rangle||^2 \; d\mu(x) < \infty, \quad \forall f \in \Hpi \bigg\}.
\end{align}
Then $\mathcal{B}^2_{\pi}$ is dense in $\Hpi$ and independent of the choice of defining neighborhood; cf. \cite{grochenig2008homogeneous}.

\begin{proposition}[\cite{grochenig2008homogeneous}] \label{prop:frame_hap}
Suppose that $g \in \mathcal{B}_{\pi}^2$ and that $\pi(\Lambda) g$ is a frame with canonical dual frame $(h_{\lambda})_{\lambda \in \Lambda}$. Then $(h_{\lambda})_{\lambda \in \Lambda}$ satisfies the homogeneous approximation property. 
\end{proposition}

\begin{theorem} \label{thm:frame_riesz}
Let $\Lambda \subseteq G$ be discrete. The following assertions hold:
\begin{enumerate}[(i)]
\item If $g \in \mathcal{B}^2_{\pi}$ and $\pi(\Lambda) g$ is a frame for $\Hpi$, then $D^- (\Lambda) \geq d_{\pi}$.
\item If $g \in \Hpi$ and $\pi(\Lambda) g$ is a Riesz sequence in $\Hpi$, then $D^+ (\Lambda) \leq d_{\pi}$. 
\end{enumerate}
\end{theorem}
\begin{proof}
(i) If $\pi(\Lambda) g$ is a frame, then $\Lambda$ is locally finite. Since $g \in \mathcal{B}_{\pi}^2$, the canonical dual frame $(h_{\lambda} )_{\lambda \in \Lambda}$ of $\pi(\Lambda) g$ satisfies the homogeneous approximation property by \Cref{prop:frame_hap}. Applying \Cref{thm:HAP_density} to $(h_{\lambda} )_{\lambda \in \Lambda}$ yields $D^- (\Lambda) \geq d_{\pi}$. 

(ii) If $\pi(\Lambda) g$ is a Riesz sequence, then it is automatically uniformly minimal, and hence $D^+ (\Lambda) \leq d_{\pi}$ by \Cref{thm:uniformly_minimal}. 
\end{proof}

Theorem \ref{thm:frame_riesz}  gives a new proof of the main result of \cite{enstad2022dynamical} for amenable unimodular groups. 

\section{Examples} \label{sec:example}
This section provides two examples that illustrate the applicability of the main theorems.

The first example provides a group in which \Cref{thm:main1} applies, but where the density theorems \cite{bekka2004square, romero2004complete, enstad2021density} for complete lattice orbits do not as the group does not admit a lattice.

\begin{example}[Abelian group without a lattice]
For a prime number $p$, define the additive subgroup $\Z[1/p] = \{ q /p^k : q \in \mathbb{Z}, k \in \mathbb{N}\}$ of $\Q$. Let $\Q_p$ denote the field of $p$-adic numbers, which is the completion of $\Q$ in the $p$-adic norm $| \cdot |_p$. Identifying $\Z[1/p]$ as a subset of $\Q_p \times \R$ using the map $q \mapsto (q,q)$, the set $\Z[1/p]$ forms a lattice in $\Q_p \times \R$. For this, note that the compact identity neighborhood $C = \{ (x,y) \in \Q_p \times \R : |x|_p \leq 1, |y| \leq 1/2 \}$ satisfies $C \cap \Z[1/p] = \{ (0,0) \}$ and $C + \Z[1/p] = \Q_p \times \R$. Since $\Z[1/p]$ projects injectively and densely to either factor, it follows that, for any compact symmetric neighborhood $W$ of the identity of $\R$, the corresponding model set 
\[ \Lambda = p_{\Q_p}(\Z[1/p] \cap (\Q_p \times W))
\]
is an approximate lattice in $\Q_p$; see \Cref{ex:model_set}. On the other hand, the group $\Q_p$ contains no discrete subgroups besides the trivial subgroup; in particular, it does not admit lattices.

The product group $G = \Q_p^2$ similarly contains approximate lattices, but no lattices. Since $\Q_p^2 \cong \Q_p \times \widehat{\Q_p}$, where $\widehat{\Q_p}$ denotes the Pontryagin dual of $\Q_p$, the projective representation $\pi$ of $G$ on $L^2(\Q_p)$ given by 
\[ \pi(x,\omega)f(t) = \omega(t)f(t-x), \qquad (x,\omega) \in \Q_p \times \widehat{\Q_p}, \, f \in L^ 2(\Q_p), \]
gives an example of a projective discrete series of $G$; see \cite{grochenig1998aspects}. 
\end{example}

The second example provides a group where \Cref{thm:frame_riesz} is applicable, but not the density conditions of \cite{fuehr2017density, mitkovski2020density} as the group has exponential growth. The example is taken from \cite{rosenberg1978square}.

\begin{example}[Solvable Lie group with exponential growth]\label{ex:exponential}
Let $\mathfrak{s} = \Rspan \{X_1, ..., X_5\}$ be the Lie algebra with
\[
[X_1, X_2] = X_3, \quad [X_5, X_1] = X_1, \quad [X_5, X_2] = -X_2, \quad [X_5, X_4] = X_3.
\]
Then $\mathfrak{s}$ is completely solvable, but not nilpotent. Its center is $\mathfrak{z} = \mathbb{R} X_3$. The simply connected Lie group $S$ with Lie algebra $\mathfrak{s}$ is unimodular and admits irreducible, square-integrable representations modulo $Z = \exp(\mathfrak{z})$. 
See, e.g., \cite[Section 4.13]{rosenberg1978square} for details.

The quotient group $G := S/Z$ admits the projective discrete series $\pi := \rho \circ s$, where $\rho$ is any irreducible, square-integrable representation (modulo $Z$) of $S$ and $s : G \to S$ a Borel cross-section. Since $G$ is an exponential solvable Lie group, its Lie algebra $\mathfrak{g}$ has the property that, for every $Y \in \mathfrak{g}$, the adjoint representation $\ad(Y)$ on  $\mathfrak{g}$ has no non-zero purely imaginary eigenvalues. As $\mathfrak{g}$ is non-nilpotent, there exists $Y \in \mathfrak{g}$ such that $\ad(Y)$ has a non-zero non-purely imaginary eigenvalue, and thus $G$ must have exponential growth by \cite[Theorem 1.4]{jenkins1973growth}.
\end{example}

\section*{Acknowledgements}
U.E.\ gratefully acknowledges support from the The Research Council of Norway through project 314048. J.v.V. gratefully acknowledges support from 
the Austrian Science Fund (FWF) project J-4555. The authors thank the referee for helpful suggestions and for pointing out an inaccuracy in an example in an earlier version of the paper. 

\bibliographystyle{abbrv}
\bibliography{bib}

\end{document}